\documentclass[a4,12pt]{article}
\title{Algebraic setup of non-strict multiple zeta values}
\author{Shuichi Muneta}
\date{}

\usepackage{amsmath,amsthm,amssymb}

\font\xiirm=cmr12 \font\xrm=wncyr10
\newcommand{\sh}{\xrm \mbox{sh}\, \xiirm}

\makeatletter

\@addtoreset{equation}{section}
\makeatother

\newcommand{\sv}{\,\overline{\sh}\,}
\newcommand{\hv}{\,\overline{*}\,}
\newcommand{\zs}{\overline{\zeta}}
\newcommand{\cv}{\,\overline{\circ}\,}

\begin{document}
\maketitle

\newtheorem{thm}{Theorem}[section]
\newtheorem{lem}[thm]{Lemma}
\newtheorem{prop}[thm]{Proposition}

\section{Introduction}
The multiple zeta values and non-strict multiple zeta values\footnote[0]{2000 {\it Mathematics Subject
Classification}. Primary 11M41.} (MZVs and NMZVs, for short) 
are defined respectively by 
\[\zeta(k_{1},k_{2},\ldots,k_{n})
 :=\sum_{m_{1}>m_{2}>\cdots>m_{n}>0}\frac{1}{m_{1}^{k_{1}}m_{2}^{k_{2}}{\cdots}m_{n}^{k_{n}}},\]
\[\zs (k_{1},k_{2},\ldots,k_{n})
 :=\sum_{m_{1} \geq m_{2} \geq \cdots \geq m_{n}>0} \frac{1} {m_{1}^{k_{1}}m_{2}^{k_{2}} \cdots m_{n}^{k_{n}}},\]
where $k_{1},k_{2},\ldots,k_{n}$ are positive integers and $k_{1}\geq 2$. 
Considerable amount of work on MZVs has been done from various aspects and interests. 

The MZVs have many relations among them 
(duality formula, sum formula, Hoffman's relations, Ohno's relations, 
derivation relations and cyclic sum relations, cf. [H1], [HO], [IKZ], [O]) 
and these relations can be described in purely algebraic manner (cf. [IKZ]). 
On the other hand, NMZVs have not been investigated so much compared to MZVs. 
But recently, a few works on NMZVs have appeared ([AO], [OW]) and they indicate that 
NMZVs possess similar properties to MZVs.  

In this article, we introduce an algebraic setup of NMZVs and prove some relations of NMZVs, 
which are analogous to Hoffman's relations of MZVs, by using this algebraic setup of NMZVs. 

\section{Algebraic setup of NMZVs}
\subsection{Algebraic setup of MZVs}
We summarize the algebraic setup of MZVs introduced by Hoffman (cf. [H2], [IKZ]). 
Let $ \mathfrak{H} = \mathbb{Q} \left\langle x,y \right\rangle$ be 
the noncommutative polynomial ring in two indeterminates $x$, $y$ 
and $\mathfrak{H}^1$ and $\mathfrak{H}^0$ its subrings $\mathbb{Q} + \mathfrak{H}y$ 
and $\mathbb{Q} + x \mathfrak{H} y$. 
We set $z_{k} = x^{k-1} y$ $(k = 1,2,3,\ldots)$. Then $\mathfrak{H}^1$ 
is freely generated by $\{z_{k}\}_{k \geq 1}$. 
For any word $w$, let $l(w)$ be the degree of $w$ with respect to $y$, 
and $\left| w \right|$ the total degree.

We define the $\mathbb{Q}$-linear map (called evaluation map) $Z:\mathfrak{H}^0 \longrightarrow \mathbb{R}$ by
\[Z(1)=1 \;\; \mathrm{and} \;\;  Z( z_{k_{1}} z_{k_{2}} \cdots z_{k_{n}}) = \zeta( k_{1}, k_{2}, \ldots, k_{n}).\]

We next define two products of MZVs. 
The one is the harmonic product $*$ on $\mathfrak{H}^1$ defined by
\begin{eqnarray*}
 1*w &=& w*1 \;=\; w , \\
 z_{k} w_{1} * z_{l} w_{2} 
 &=& z_{k} (w_{1} * z_{l} w_{2}) + z_{l} (z_{k} w_{1} * w_{2}) + z_{k+l} (w_{1} * w_{2}) 
\end{eqnarray*}
($k,l \in \mathbb{Z}_{\geq 1}$ and $w$, $w_{1}$, $w_{2}$ are words in $\mathfrak{H}^1$), 
together with $\mathbb{Q}$-bilinearity. 
The harmonic product $*$ is commutative and associative, 
therefore $\mathfrak{H}^1$ is $\mathbb{Q}$-commutative algebra with respect to $*$. 
We denote it by $\mathfrak{H}^{1}_{*}$. 
The subset $\mathfrak{H}^0$ is a subalgebra of $\mathfrak{H}^1$ 
with respect to $*$ and we denote it by $\mathfrak{H}^{0}_{*}$. 
We then have 
\[Z(w_{1} * w_{2}) = Z(w_{1}) Z(w_{2}) \]
for any $w_{1}, w_{2} \in \mathfrak{H}^0$. 
The other product is the shuffle product $\sh$ on $\mathfrak{H}$ defined by
\begin{eqnarray*}
 1 \sh w &=& w \sh 1 \;=\; w , \\
 u_{1} w_{1} \sh u_{2} w_{2} &=& u_{1} (w_{1} \sh u_{2} w_{2}) + u_{2} (u_{1} w_{1} \sh w_{2})  
\end{eqnarray*}
($u_{1},u_{2} \in \{ x,y \}$ and $w$, $w_{1}$, $w_{2}$ are words in $\mathfrak{H}$), together with $\mathbb{Q}$-bilinearity. 
The shuffle product $\sh$ is also commutative and associative, 
therefore $\mathfrak{H}$ is $\mathbb{Q}$-commutative algebra with respect to $\sh$. 
We denote it by $\mathfrak{H}_{\sh}$. The subsets $\mathfrak{H}^1$ and $\mathfrak{H}^0$ 
are subalgebras of $\mathfrak{H}$ with respect to $\sh$ 
and we denote them by $\mathfrak{H}^{1}_{\sh}$, $\mathfrak{H}^{0}_{\sh}$ respectively. 
For this product, we also have 
\[ Z(w_{1} \sh w_{2}) = Z(w_{1}) Z(w_{2}) \]
for any $w_{1}, w_{2} \in \mathfrak{H}^0$.

The finite double shuffle relations of MZVs is then 
\[Z( w_{1} * w_{2} - w_{1} \sh w_{2}) = 0 \quad  (w_{1}, w_{2} \in \mathfrak{H}^0).\]

The evaluation map is generalized by the following proposition to get the extended double shuffle relations of MZVs. 

\begin{prop}[{[}IKZ{]}]
We have two algebra homomorphisms
\[ Z^{*}:\mathfrak{H}^{1}_{*} \longrightarrow \mathbb{R}[T] 
 \quad and \quad
   Z^{\sh}:\mathfrak{H}^{1}_{\sh} \longrightarrow \mathbb{R}[T]\]
which are uniquely characterized by the properties that they both extend the evaluation map 
$Z:\mathfrak{H}^0 \longrightarrow \mathbb{R}$ and send $y$ to $T$.
\end{prop}

Then we have the extended double shuffle relations of MZVs. 
\begin{thm}[{[}IKZ{]}]
For any $w_{1} \in \mathfrak{H}^1$ and $w_{2} \in \mathfrak{H}^0$, we have 
\[Z^{*} ( w_{1} \sh w_{2} - w_{1} * w_{2} ) = 0
 \quad and \quad 
 Z^{\sh} ( w_{1} \sh w_{2} - w_{1} * w_{2} ) = 0.\]
\label{2.2}
\end{thm}

\subsection{Algebraic setup of NMZVs}
In this subsection, we introduce the algebraic setup of NMZVs. 
Define $\mathbb{Q}$-linear map $\overline{Z}:\mathfrak{H}^0 \longrightarrow \mathbb{R}$ by
\[ \overline{Z}(1) = 1 
\quad \mathrm{and} \quad 
\overline{Z}(z_{k_{1}} z_{k_{2}} \cdots z_{k_{n}}) = \zs (k_{1}, k_{2}, \ldots, k_{n}).\]
We call this map n-evaluation map. 
We next define the n-harmonic product $\hv$ on $\mathfrak{H}^1$, 
which is the NMZV-couterpart of the harmonic product $*$, inductively by
\begin{eqnarray*}
 1 \hv w &=& w \hv 1 \;=\; w , \\
 z_{k} w_{1} \hv z_{l} w_{2} &=& z_{k} (w_{1} \hv z_{l} w_{2}) + z_{l} (z_{k} w_{1} \hv w_{2}) - z_{k+l} (w_{1} \hv w_{2}) 
\end{eqnarray*}
($k,l \in \mathbb{Z}_{\geq 1}$ and $w$, $w_{1}$, $w_{2}$ are words in $\mathfrak{H}^1$), together with $\mathbb{Q}$-bilinearity. 
The n-harmonic product $\hv$ has the following properties. 

\begin{prop}\label{2.3}
The n-harmonic product $\hv$ is commutative and associative. 
\end{prop}
\begin{proof}
We can prove the assertion by induction. (cf. Theorem 2.1 of [H2].) But we later give another proof.  
\end{proof}

Proposition \ref{2.3} says that $\mathfrak{H}^1$ has the commutative $\mathbb{Q}$-algebra structure 
with respect to $\hv$. 
We denote this algebra by $\mathfrak{H}^{1}_{\hv}$. The subset $\mathfrak{H}^0$ is a subalgebra of $\mathfrak{H}^1$ 
with respect to $\hv$ 
and we denote it by $\mathfrak{H}^{0}_{\hv}$. 

We introduce the $\mathbb{Q}$-linear map $S:\mathfrak{H}^1 \longrightarrow  \mathfrak{H}^1$. 
Let $S_{1} \in Aut(\mathfrak{H})$ be defined by $S_{1}(1)=1$, $S_{1}(x)=x$ and $S_{1}(y)=x+y$. 
Define the $\mathbb{Q}$-linear map $S$ : $\mathfrak{H}^{1} \longrightarrow \mathfrak{H}^{1}$ by 
\[S(1):=1 \quad \mathrm{and} \quad S(Fy):=S_{1}(F)y\]
for all words $F \in \mathfrak{H}$. 
Then it is clear that $\overline{Z} = Z \circ S$ on $\mathfrak{H}^0$, i.e., 
\[\zs(k_{1},k_{2},\ldots,k_{n}) = Z (S( z_{k_{1}} z_{k_{2}} \cdots z_{k_{n}} ) ) \quad (k_{1} \geq 2).\]
For example, 
$\zs(k_{1},k_{2}) = \zeta(k_{1} + k_{2}) + \zeta(k_{1} , k_{2} ) 
 = Z( S( z_{k_{1}} z_{k_{2}} ) )$, 
$\zs(k_{1},k_{2},k_{3}) = \zeta(k_{1} + k_{2} + k_{3}) + \zeta(k_{1} + k_{2} , k_{3}) 
 + \zeta(k_{1} , k_{2} + k_{3}) + \zeta(k_{1} , k_{2} , k_{3}) 
 = Z( S( z_{k_{1}} z_{k_{2}} z_{k_{3}} ) )$. 
As is clear from the definition of $S$, we also have the following relation: 
\begin{align}
 S( w_{1} w_{2}) &= S_{1}(w_{1}) S(w_{2}) 
 \quad (w_{1} \in \mathfrak{H}, w_{2} \in \mathfrak{H}^1). \label{eq:2.1}
\end{align}

\begin{prop}\label{2.4}
For $w_{1}$, $w_{2} \in \mathfrak{H}^0$, we have 
\[ \overline{Z}(w_{1} \hv w_{2}) = \overline{Z}(w_{1}) \overline{Z}(w_{2}). \]
\end{prop}

To prove Proposition \ref{2.4}, we need the following lemma. 
\begin{lem}\label{2.5}
Let $w$, $w_{1}$, $w_{2}$ be words $($$\neq 1$$)$ in $\mathfrak{H}^1$ and $p$, $q$ positive integers. 
Then we have 
\begin{align}
\qquad\; S(z_{p}) * S_{1}(z_{q})w &=  S_{1}(z_{p} z_{q})w + S_{1}(z_{q})(S(z_{p}) * w) - S_{1}(z_{p+q})w \label{eq:2.2} 
\end{align}
and
\begin{align}
S_{1}(z_{p})w_{1} * S_{1}(z_{q})w_{2} &= S_{1}(z_{p})(w_{1} * S_{1}(z_{q})w_{2}) + S_{1}(z_{q})(S_{1}(z_{p})w_{1} * w_{2})\nonumber \\
 & \quad - S_{1}(z_{p+q})(w_{1} * w_{2}) \label{eq:2.3}
\end{align}
\end{lem}
\begin{proof}
We first prove (\ref{eq:2.2}). Put $w = z_{n} \tilde{w} \, (n \geq 1, \tilde{w} \in \mathfrak{H}^{1})$, then 
\begin{align*}
\lefteqn{ \mathrm{RHS \; of \;} (\ref{eq:2.2}) } \\
 &= (x^{p-1}y + x^{p}) (x^{q-1}y + x^{q}) z_{n} \tilde{w} 
   + (x^{q-1}y + x^{q}) (z_{p} * z_{n} \tilde{w}) \displaybreak[0]\\
 & \quad - (x^{p+q-1}y + x^{p+q}) z_{n} \tilde{w} \displaybreak[0]\\
 &= z_{p} z_{q} z_{n} \tilde{w} + z_{p} z_{n+q} \tilde{w} + z_{p+q} z_{n} \tilde{w} + z_{n+p+q} \tilde{w} \displaybreak[0]\\
 & \quad  + z_{q} z_{p} z_{n} \tilde{w} + z_{q} z_{n} (z_{p} * \tilde{w}) + z_{q} z_{n+p} \tilde{w} 
           + z_{p+q} z_{n} \tilde{w} + z_{n+q} (z_{p} * \tilde{w}) + z_{n+p+q} \tilde{w} \\
 & \quad - z_{p+q} z_{n} \tilde{w} -z_{n+p+q} \tilde{w} \displaybreak[0]\\  
 &= z_{p} * z_{q} z_{n} \tilde{w} + z_{p} * z_{n+q} \tilde{w} 
  = S(z_{p}) * S_{1}(z_{q}) z_{n} \tilde{w} 
  = S(z_{p}) * S_{1}(z_{q}) w.
\end{align*}
Hence (\ref{eq:2.2}) follows. 
Putting $w_{1} = z_{m} \tilde{w_{1}}$, $w_{2} = z_{n} \tilde{w_{2}} \, (m, n \geq 1, \tilde{w_{1}}, \tilde{w_{2}} \in \mathfrak{H}^{1})$, 
we can prove (\ref{eq:2.3}) in the same way.  
\end{proof}
\begin{proof}[ Proof of Proposition $\ref{2.4}$] 
It suffices to show that 
\begin{align}
S(w_{1} \hv w_{2}) &= S(w_{1}) * S(w_{2}) \label{eq:2.4}
\end{align}
for words $w_{1}$, $w_{2} \in \mathfrak{H}^1$. 
We set $w_{1} = z_{ p_{1} } z_{ p_{2} } \cdots z_{ p_{m} }$, 
$w_{2} = z_{ q_{1} } z_{ q_{2} } \cdots z_{ q_{n} }$. 
We prove (\ref{eq:2.4}) by induction on $m$. To ease the following calculation, we set 
$z_{\vec{p}} = z_{ p_{2} } z_{ p_{3} }\cdots z_{ p_{m} }$ and $z_{\vec{q}} = z_{ q_{2} } z_{ q_{3} }\cdots z_{ q_{n} }$. 
(i) We prove the case $m=1$ by induction on $n$. When $n=1$, the assertion is immediate. 
We assume the assertion is proven for $n-1$. Using (\ref{eq:2.1}), (\ref{eq:2.2}) and the induction hypothesis, we have 
\begin{align*}
\lefteqn{S(z_{p_{1}} \hv z_{q_{1}} z_{q_{2}} \cdots z_{q_{n}}) 
          = S(z_{p_{1}} \hv z_{q_{1}} z_{\vec{q}}) } \\
&= S \big( z_{p_{1}} z_{q_{1}} z_{\vec{q}} 
     + z_{q_{1}} (z_{p_{1}} \hv z_{\vec{q}}) 
     - z_{p_{1}+q_{1}} z_{\vec{q}}  \big) \displaybreak[0]\\
&= S_{1}( z_{p_{1}} z_{q_{1}}) S( z_{\vec{q}} )
     + S_{1}( z_{q_{1}} ) S(z_{p_{1}} \hv z_{\vec{q}} ) 
     - S_{1}( z_{p_{1}+q_{1}} ) S( z_{\vec{q}} )  \displaybreak[0]\\
&= S_{1}( z_{p_{1}} z_{q_{1}}) S( z_{\vec{q}} )
     + S_{1}( z_{q_{1}} )\big( S(z_{p_{1}}) * S( z_{\vec{q}} ) \big)
     - S_{1}( z_{p_{1}+q_{1}} ) S( z_{\vec{q}} )  \displaybreak[0]\\
&= S( z_{p_{1}}) * S_{1}(z_{q_{1}}) S( z_{\vec{q}}) \displaybreak[0]\\
&= S( z_{p_{1}}) * S(z_{q_{1}} z_{\vec{q}}) = S( z_{p_{1}}) * S(z_{q_{1}} z_{q_{2}} \cdots z_{q_{n}} ). 
\end{align*}
(ii) We assume the assertion is proven for $m-1$. We prove the assertion for $m$ by induction on $n$. 
When $n=1$, the assertion follows from (i) and the commutativity of $*$, $\hv$. 
We assume the assertion is true for $n-1$. 
Using $(\ref{eq:2.1})$, $(\ref{eq:2.3})$ and the induction hypothesis, we have
\begin{align*}
\lefteqn{S(z_{p_{1}} z_{p_{2}} \cdots z_{p_{m}} \hv z_{q_{1}} z_{q_{2}} \cdots z_{q_{n}}) 
          = S(z_{p_{1}} z_{\vec{p}} \hv z_{q_{1}} z_{\vec{q}}) } \\
&= S \big( z_{p_{1}} ( z_{\vec{p}} \hv z_{q_{1}} z_{\vec{q}})
     + z_{q_{1}} (z_{p_{1}} z_{\vec{p}} \hv z_{\vec{q}}) 
     - z_{p_{1}+q_{1}} ( z_{\vec{p}} \hv z_{\vec{q}} ) \big) \displaybreak[0]\\
&= S_{1}( z_{p_{1}} ) S( z_{\vec{p}} \hv z_{q_{1}} z_{\vec{q}} )
     + S_{1}( z_{q_{1}} ) S(z_{p_{1}} z_{\vec{p}} \hv z_{\vec{q}} ) 
     - S_{1}( z_{p_{1}+q_{1}} ) S( z_{\vec{p}} \hv z_{\vec{q}} )  \displaybreak[0]\\
&= S_{1} (z_{p_{1}}) \big( S( z_{\vec{p}} ) * S(z_{q_{1}} z_{\vec{q}} ) \big)
   +S_{1} (z_{q_{1}}) \big( S(z_{p_{1}} z_{\vec{p}}) * S( z_{\vec{q}} ) \big) \displaybreak[0]\\
& \quad -S_{1} (z_{p_{1}+q_{1}}) \big( S( z_{\vec{p}} ) * S( z_{\vec{q}} ) \big) \displaybreak[0]\\
&= S_{1} (z_{p_{1}}) \big( S( z_{\vec{p}} ) * S_{1}(z_{q_{1}})S( z_{\vec{q}} ) \big)
   +S_{1} (z_{q_{1}}) \big( S_{1}(z_{p_{1}})S( z_{\vec{p}} ) * S( z_{\vec{q}} ) \big) \\
& \quad -S_{1} (z_{p_{1}+q_{1}}) \big( S( z_{\vec{p}} ) * S( z_{\vec{q}} ) \big) \displaybreak[0]\\
&= S_{1}(z_{p_{1}}) S( z_{\vec{p}} ) 
   * S_{1}(z_{q_{1}}) S( z_{\vec{q}} ) \\
&= S(z_{p_{1}} z_{\vec{p}} ) * S(z_{q_{1}} z_{\vec{q}} ) 
 = S(z_{p_{1}} z_{p_{2}} \cdots z_{p_{m}}) * S(z_{q_{1}} z_{q_{2}} \cdots z_{q_{n}}).
\end{align*}
This completes the proof. 
\end{proof}

We shall define the n-shuffle product $\sv$ on $\mathfrak{H}$ which corresponds to the shuffle product $\sh$. 
The n-shuffle product $\sv$ is defined inductively by 
\begin{align*}
1 \sv w & = w \sv 1 = w   \\
u_{1} w_{1} \sv u_{2} w_{2} 
 &= u_{1} (w_{1} \sv u_{2} w_{2}) + u_{2} (u_{1} w_{1} \sv w_{2}) \\
 & \qquad - \delta(w_{1}) \tau(u_{1}) u_{2} w_{2} - \delta(w_{2}) \tau(u_{2}) u_{1} w_{1} 
\end{align*}
($u_{1}, u_{2} \in \{x,y\}$ and $w$, $w_{1}$, $w_{2}$ are words in $\mathfrak{H}$), together with $\mathbb{Q}$-bilinearity, 
where $\delta$ is defined by 
\[\delta(w)=
\left\{
\begin{array}{ll}
 1 &\quad  (w=1), \\
 0 &\quad  (w\neq 1)
\end{array}
\right.\]
for word $w$ and $\tau$ is defined by $\tau(x) = y$, $\tau(y) = x$. 
The n-shuffle product has the following properties. 
\begin{prop}
The n-shuffle product is commutative and associative. 
\label{2.6}
\end{prop}
\begin{proof}
Let $w_{1}$, $w_{2}$, $w_{3}$ be words in $\mathfrak{H}$. 
We can check the commutativity $w_{1} \sv w_{2} = w_{2} \sv w_{1}$ by induction on $\left| w_{1} \right| + \left| w_{2} \right|$. 
We shall prove the associativity $(w_{1} \sv w_{2}) \sv w_{3} = w_{1} \sv (w_{2} \sv w_{3})$ 
by induction on $\left| w_{1} \right| + \left| w_{2} \right| + \left| w_{3} \right|$. 
The case $\left| w_{1} \right| + \left| w_{2} \right| + \left| w_{3} \right| \leq 2$ is obvious. 
Putting $w_{1} = u_{1} \tilde{ w_{1} }$, $w_{2} = u_{2} \tilde{ w_{2} }$, $w_{3} = u_{3} \tilde{ w_{3} }$
($u_{1}, u_{2}, u_{3} \in \left\{x,y \right\}$), we have 
\begin{align*}
\lefteqn{(w_{1} \sv w_{2}) \sv w_{3}} \\
 &= u_{1} (\tilde{ w_{1} } \sv u_{2} \tilde{ w_{2} }) \sv u_{3} \tilde{ w_{3} } 
    + u_{2} (u_{1} \tilde{ w_{1} } \sv \tilde{ w_{2} }) \sv u_{3} \tilde{ w_{3} } \\
 & \quad  - \delta(\tilde{ w_{1} }) \tau(u_{1}) u_{2} \tilde{ w_{2} } \sv u_{3} \tilde{ w_{3} } 
           - \delta(\tilde{ w_{2} }) \tau(u_{2}) u_{1} \tilde{ w_{1} } \sv u_{3} \tilde{ w_{3} } \displaybreak[0]\\
 &= u_{1} \left\{ (\tilde{ w_{1} } \sv u_{2} \tilde{ w_{2} }) \sv u_{3} \tilde{ w_{3} } \right\} 
    + u_{3} \left\{ u_{1} (\tilde{ w_{1} } \sv u_{2} \tilde{ w_{2} }) \sv \tilde{ w_{3} } \right\} \\
 & \quad - \delta(\tilde{ w_{3} }) \tau(u_{3}) u_{1} (\tilde{ w_{1} } \sv u_{2} \tilde{ w_{2} }) 
         + u_{2} \left\{ (u_{1} \tilde{ w_{1} } \sv \tilde{ w_{2} }) \sv u_{3} \tilde{ w_{3} } \right\} \\
 & \quad  + u_{3} \left\{ u_{2} (u_{1} \tilde{ w_{1} } \sv \tilde{ w_{2} }) \sv \tilde{ w_{3} }\right\}
          -\delta(\tilde{ w_{3} }) \tau(u_{3}) u_{2} (u_{1} \tilde{ w_{1} } \sv \tilde{ w_{2} }) \\
 & \quad - \delta(\tilde{ w_{1} }) \tau(u_{1}) u_{2} \tilde{ w_{2} } \sv u_{3} \tilde{ w_{3} } 
         - \delta(\tilde{ w_{2} }) \tau(u_{2}) u_{1} \tilde{ w_{1} } \sv u_{3} \tilde{ w_{3} } \displaybreak[0]\\
 &= u_{1} \left\{ (\tilde{ w_{1} } \sv u_{2} \tilde{ w_{2} }) \sv u_{3} \tilde{ w_{3} } \right\}
    + u_{2} \left\{ (u_{1} \tilde{ w_{1} } \sv \tilde{ w_{2} }) \sv u_{3} \tilde{ w_{3} } \right\}   \\
 & \quad + u_{3} \left\{ (u_{1} \tilde{ w_{1} } \sv u_{2} \tilde{ w_{2} }) \sv \tilde{ w_{3} } \right\} 
   - \delta(\tilde{ w_{1} }) \tau(u_{1}) (u_{2} \tilde{ w_{2} } \sv u_{3} \tilde{ w_{3} }) \\
 & \quad - \delta(\tilde{ w_{2} }) \tau(u_{2}) (u_{1} \tilde{ w_{1} } \sv u_{3} \tilde{ w_{3} })
     - \delta(\tilde{ w_{3} }) \tau(u_{3}) (u_{1} \tilde{ w_{1} } \sv u_{2} \tilde{ w_{2} }).
\end{align*} 
In the last equality, we use the following three relations: 
\begin{align*}
\lefteqn{ u_{1} (\tilde{ w_{1} } \sv u_{2} \tilde{ w_{2} }) + u_{2} (u_{1} \tilde{ w_{1} } \sv \tilde{ w_{2} }) } \\
 & \quad = u_{1} \tilde{ w_{1} } \sv u_{2} \tilde{ w_{2} } + \delta(\tilde{ w_{1} }) \tau(u_{1}) u_{2} \tilde{ w_{2} }
                                                     + \delta(\tilde{ w_{2} }) \tau(u_{2}) u_{1} \tilde{ w_{1} },\\
\lefteqn{\tau(u_{1}) u_{2} \tilde{ w_{2} } \sv u_{3} \tilde{ w_{3} }} \\
 & \quad =  \tau(u_{1}) (u_{2} \tilde{ w_{2} } \sv u_{3} \tilde{ w_{3} }) + u_{3} (\tau(u_{1}) u_{2} \tilde{ w_{2} } \sv \tilde{ w_{3} })
          - \delta(\tilde{ w_{3} }) \tau(u_{3}) \tau(u_{1}) u_{2} \tilde{ w_{2} }, \\
\lefteqn{\tau(u_{2}) u_{1} \tilde{ w_{1} } \sv u_{3} \tilde{ w_{3} } } \\
 & \quad =  \tau(u_{2}) (u_{1} \tilde{ w_{1} } \sv u_{3} \tilde{ w_{3} }) + u_{3} (\tau(u_{2}) u_{1} \tilde{ w_{1} } \sv \tilde{ w_{3} })
          - \delta(\tilde{ w_{3} }) \tau(u_{3}) \tau(u_{2}) u_{1} \tilde{ w_{1} }.
\end{align*}
On the other hand, 
\begin{align*}
\lefteqn{w_{1} \sv (w_{2} \sv w_{3})} \\
 &= u_{1} \tilde{ w_{1} } \sv u_{2} (\tilde{ w_{2} } \sv u_{3} \tilde{ w_{3} } ) 
    + u_{1} \tilde{ w_{1} } \sv u_{3} (u_{2} \tilde{ w_{2} } \sv  \tilde{ w_{3} }) \\
 & \quad - \delta(\tilde{ w_{2} })  u_{1} \tilde{ w_{1} } \sv \tau(u_{2}) u_{3} \tilde{ w_{3} } 
         - \delta(\tilde{ w_{3} }) u_{1} \tilde{ w_{1} } \sv \tau(u_{3}) u_{2} \tilde{ w_{2} } \displaybreak[0]\\
 &= u_{1} \left\{ \tilde{ w_{1} } \sv u_{2} (\tilde{ w_{2} } \sv u_{3} \tilde{ w_{3} } ) \right\}
    + u_{2} \left\{ u_{1} \tilde{ w_{1} } \sv (\tilde{ w_{2} } \sv u_{3} \tilde{ w_{3} } ) \right\} \\
 & \quad - \delta(\tilde{ w_{1} }) \tau(u_{1}) u_{2} (\tilde{ w_{2} } \sv u_{3} \tilde{ w_{3} } ) 
   + u_{1} \left\{ \tilde{ w_{1} } \sv u_{3} (u_{2} \tilde{ w_{2} } \sv  \tilde{ w_{3} }) \right\} \\ 
 & \quad + u_{3} \left\{ u_{1} \tilde{ w_{1} } \sv (u_{2} \tilde{ w_{2} } \sv  \tilde{ w_{3} }) \right\} 
          - \delta(\tilde{ w_{1} }) \tau(u_{1}) u_{3} (u_{2 }\tilde{ w_{2} } \sv \tilde{ w_{3} } ) \\
 & \quad - \delta(\tilde{ w_{2} }) u_{1} \tilde{ w_{1} } \sv \tau(u_{2}) u_{3} \tilde{ w_{3} }  
          - \delta(\tilde{ w_{3} }) u_{1} \tilde{ w_{1} } \sv \tau(u_{3}) u_{2} \tilde{ w_{2} } \displaybreak[0]\\
 &= u_{1} \left\{ \tilde{ w_{1} } \sv (u_{2} \tilde{ w_{2} } \sv u_{3} \tilde{ w_{3} }) \right\} 
    + u_{2} \left\{ u_{1} \tilde{ w_{1} } \sv (\tilde{ w_{2} } \sv u_{3} \tilde{ w_{3} }) \right\} \\
 & \quad  + u_{3} \left\{ u_{1} \tilde{ w_{1} } \sv (u_{2} \tilde{ w_{2} } \sv \tilde{ w_{3} }) \right\} 
  -\delta(\tilde{ w_{1} }) \tau(u_{1}) (u_{2} \tilde{ w_{2} } \sv u_{3} \tilde{ w_{3} }) \\
 & \quad -\delta(\tilde{ w_{2} }) \tau(u_{2}) (u_{1} \tilde{ w_{1} } \sv u_{3} \tilde{ w_{3} }) 
           -\delta(\tilde{ w_{3} }) \tau(u_{3}) (u_{1} \tilde{ w_{1} } \sv u_{2} \tilde{ w_{2} }).
\end{align*}
In the last equality, we use the following three relations: 
\begin{align*}
\lefteqn{u_{2} (\tilde{ w_{2} } \sv u_{3} \tilde{ w_{3} } ) + u_{3} (u_{2} \tilde{ w_{2} } \sv  \tilde{ w_{3} }) } \\
 & \quad = u_{2} \tilde{ w_{2} } \sv u_{3} \tilde{ w_{3} } 
         + \delta(\tilde{ w_{2} }) \tau(u_{2}) u_{3} \tilde{ w_{3} } 
         + \delta(\tilde{ w_{3} }) \tau(u_{3}) u_{2} \tilde{ w_{2} }, \\
\lefteqn{u_{1} \tilde{ w_{1} } \sv \tau(u_{2}) u_{3} \tilde{ w_{3} } } \\
 & \quad = u_{1} (\tilde{ w_{1} } \sv \tau(u_{2}) u_{3} \tilde{ w_{3} }) 
         + \tau(u_{2}) (u_{1} \tilde{ w_{1} } \sv u_{3} \tilde{ w_{3} }) 
         - \delta(\tilde{ w_{1} }) \tau(u_{1}) \tau(u_{2}) u_{3} \tilde{ w_{3} }, \\
\lefteqn{u_{1} \tilde{ w_{1} } \sv \tau(u_{3}) u_{2} \tilde{ w_{2} }} \\
 & \quad = u_{1} (\tilde{ w_{1} } \sv \tau(u_{3}) u_{2} \tilde{ w_{2} }) 
         + \tau(u_{3}) (u_{1} \tilde{ w_{1} } \sv u_{2} \tilde{ w_{2} })
         - \delta(\tilde{ w_{1} }) \tau(u_{1}) \tau(u_{3}) u_{2} \tilde{ w_{2} }. 
\end{align*}
So we have the assertion by the induction hypothesis. 
\end{proof}

Proposition \ref{2.6} says that $\mathfrak{H}$ has the commutative $\mathbb{Q}$-algebra structure with respect to $\sv$. 
We denote it by $\mathfrak{H}_{\sv}$. Subsets $\mathfrak{H}^{1}$ and $\mathfrak{H}^{0}$ are subalgebras of $\mathfrak{H}$ 
with respect to $\sv$ and we denote them by $\mathfrak{H}^{1}_{\sv}$, $\mathfrak{H}^{0}_{\sv}$ respectively. 

\begin{prop}\label{2.7}
For $w_{1}$, $w_{2} \in \mathfrak{H}^0$, we have 
\[ \overline{Z}(w_{1} \sv w_{2}) = \overline{Z}(w_{1}) \overline{Z}(w_{2}). \]
\end{prop}
\begin{proof}
It suffices to prove that 
\begin{align}
 S(w_{1} \sv w_{2}) &= S(w_{1}) \sh S(w_{2}) \label{eq:2.5} 
\end{align}
for words $w_{1}$, $w_{2} \in \mathfrak{H}^1$. 
We put $w_{1} = u_{1} u_{2} \cdots u_{m}$ 
and $w_{2} = v_{1} v_{2} \cdots v_{n}$ $(u_{i}, v_{i} \in \left\{x,y \right\})$. 
We prove (\ref{eq:2.5}) by induction on $m$. 
In order to simplify the proof, we set 
$u_{\vec{m}} := u_{2} u_{3} \cdots u_{m}$ and $v_{\vec{n}} := v_{2} v_{3} \cdots v_{n}$. 
(i) We prove the case $m=1$ by induction on $n$. 
\begin{align*}
S(u_{1} \sv v_{1}) &= S(y \sv y) = S(2y^2 -2xy) = 2(x+y)y-2xy =2y^2 \\
                   &=y \sh y =S(y) \sh S(y) = S(u_{1}) \sh S(v_{1}).
\end{align*}
So the case $n=1$ is valid. We assume the assertion is proven for $n-1$. 
Using $(\ref{eq:2.1})$ and the induction hypothesis, we have
\begin{align*}
\lefteqn{S(u_{1} \sv v_{1} v_{2} \cdots v_{n}) = S(y \sv v_{1} v_{\vec{n}})} \displaybreak[0]\\
 &= S \big( y v_{1} v_{\vec{n}} + v_{1} (y \sv v_{\vec{n}}) 
  - x v_{1} v_{\vec{n}} \big) \displaybreak[0]\\
 &= S_{1}(y) S_{1}(v_{1}) S(v_{\vec{n}}) + S_{1} (v_{1}) S(y \sv v_{\vec{n}}) 
  - S_{1}(x) S_{1}(v_{1}) S(v_{\vec{n}})  \displaybreak[0]\\
 &= (x+y) S_{1}(v_{1}) S(v_{\vec{n}}) + S_{1} (v_{1}) \big( S(y) \sh S(v_{\vec{n}})) 
  - x S_{1}(v_{1}) S(v_{\vec{n}} )  \displaybreak[0]\\
 &= y S_{1} (v_{1}) S(v_{\vec{n}}) + S_{1} (v_{1}) \big( y \sh S(v_{\vec{n}}) \big)  \displaybreak[0]\\
 &= y \sh  S_{1} (v_{1}) S(v_{\vec{n}}) \displaybreak[0]\\
 &= S(u_{1}) \sh S(v_{1} v_{\vec{n}}) = S(u_{1}) \sh S(v_{1} v_{2} \cdots v_{n}).
\end{align*}
Therefore we have the assertion for $n$. 
(ii) We assume the assertion is proven for $m-1$. We prove the assertion for $m$ by induction on $n$. 
The case $n=1$ is obvious by (i) and the commutativity of $\sh$, $\sv$. 
We assume the assertion is true for $n-1$. 
\begin{align*}
\lefteqn{S(u_{1} u_{2} \cdots u_{m} \sv v_{1} v_{2} \cdots v_{n}) 
         = S(u_{1} u_{\vec{m}} \sv v_{1} v_{\vec{n}}) } \\
 &= S \big( u_{1} ( u_{\vec{m}} \sv v_{1} v_{\vec{n}} ) 
  + v_{1} (u_{1} u_{\vec{m}} \sv v_{\vec{n}}) \big)  \displaybreak[0]\\
 &= S_{1} (u_{1}) S( u_{\vec{m}} \sv v_{1} v_{\vec{n}} )   
  + S_{1} (v_{1}) S( u_{1} u_{\vec{m}}  \sv v_{\vec{n}} ) \displaybreak[0]\\
 &= S_{1} (u_{1}) \big( S(u_{\vec{m}}) \sh S(v_{1} v_{\vec{n}}) \big)  
  + S_{1} (v_{1}) \big(S(u_{1} u_{\vec{m}}) \sh S(v_{\vec{n}}) \big) \displaybreak[0]\\
 &= S_{1} (u_{1}) \big( S(u_{\vec{m}}) \sh S_{1}(v_{1}) S( v_{\vec{n}}) \big) 
  + S_{1} (v_{1}) \big( S_{1} (u_{1}) S(u_{\vec{m}}) \sh S(v_{\vec{n}}) \big) \displaybreak[0]\\
 &= S_{1} (u_{1}) S( u_{\vec{m}} ) \sh S_{1} (v_{1}) S( v_{\vec{n}} )  \displaybreak[0]\\
 &= S(u_{1} u_{\vec{m}}) \sh S(v_{1} v_{\vec{n}}) 
  = S(u_{1} u_{2} \cdots u_{m}) \sh S(v_{1} v_{2} \cdots v_{n}). 
\end{align*}
This completes the proof. 
\end{proof}
Because the n-evaluation map $\overline{Z}$ is homomorphism with 
respect to $\hv$ and $\sv$, we have the following theorem. 
\begin{thm}[Finite double shuffle relations of NMZVs]
For $w_{1}, w_{2} \in \mathfrak{H}^0$, we have 
\[\overline{Z} (w_{1} \hv w_{2} - w_{1} \sv w_{2}) = 0. \]
\label{2.8}
\end{thm}

\subsection{Extended double shuffle relations of NMZVs}
In this subsection, we generalize Theorem \ref{2.8}. 
In the following lemma, we introduce the inverse of $S$. 
\begin{lem}\label{2.9}
$(i)$ Let $S_{2} \in \mathfrak{H}$ be defined by $S_{2} (1) = 1$, $S_{2} (x) = x$ and $S_{2} (y) = y-x$. 
And we define $\mathbb{Q}$-linear map $\tilde{S}:\mathfrak{H}^1 \longrightarrow \mathfrak{H}^1$ by 
\[\tilde{S}(1) = 1 \quad and \quad \tilde{S}(Fy) := S_{2}(F)y \]
for words $F \in \mathfrak{H}$. 
Then we have $\tilde{S} \circ S = S \circ \tilde{S} = id$ on $\mathfrak{H}^1$. 

$(ii)$ For $w_{1}$, $w_{2} \in \mathfrak{H}^1$, we have 
\begin{align*}
 \tilde{S} (w_{1} * w_{2}) &= \tilde{S} (w_{1}) \hv \tilde{S} (w_{2}),  \\
 \tilde{S} (w_{1} \sh w_{2}) &= \tilde{S} (w_{1}) \sv \tilde{S} (w_{2}).
\end{align*}
\end{lem}
\begin{proof}
(i) By definition, we have $\tilde{S} \circ S (1) = S \circ \tilde{S} (1) = 1$. 
Let $w$ be contained in $\mathfrak{H}^1\backslash \left\{ 1 \right\}$. 
Then we can write $w=w_{1} y \,(w_{1} \in \mathfrak{H})$, and we have 
\[\tilde{S} \circ S (w) = \tilde{S} \circ S (w_{1} y) 
 = \tilde{S} (S_{1} (w_{1}) y ) = S_{2} (S_{1} (w_{1}) ) y = w_{1} y = w,  \]
\[ S \circ \tilde{S} (w) = S \circ \tilde{S} (w_{1} y) 
 = S (S_{2} (w_{1}) y ) = S_{1} (S_{2} (w_{1}) ) y = w_{1} y= w. \]
This completes the proof of (i). 
(ii) This is clear from (\ref{eq:2.4}), (\ref{eq:2.5}) and (i). 
\end{proof}
By (i) of Lemma \ref{2.9}, we can rewrite $\tilde{S}$ by $S^{-1}$. 
Then (ii) of Lemma \ref{2.9} can be restated as follows: 
\begin{align}
 S^{-1} (w_{1} * w_{2}) &= S^{-1} (w_{1}) \hv S^{-1} (w_{2}) , \label{eq:2.6} \\
 S^{-1} (w_{1} \sh w_{2}) &= S^{-1} (w_{1}) \sv S^{-1} (w_{2}). \label{eq:2.7} 
\end{align}
Using (\ref{eq:2.6}), we give the proof of Proposition \ref{2.3}
\begin{proof}[Proof of Proposition $\ref{2.3}$]
By using (\ref{eq:2.6}) and the commutativity of the harmonic product $*$, we have 
\begin{align*}
 w_{1} \hv w_{2} &= S^{-1} \big( S( w_{1} ) \big) \hv S^{-1} \big( S( w_{2} ) \big) 
                  = S^{-1} \big( S( w_{1} ) * S( w_{2} ) \big) \\
                 &= S^{-1} \big( S( w_{2} ) * S( w_{1} ) \big) 
                  = S^{-1} \big( S( w_{2} ) \big) \hv S^{-1} \big( S( w_{1} ) \big) 
                  = w_{2} \hv w_{1}. 
\end{align*}
So the commutativity of n-harmonic prodct $\hv$ follows. 
We next prove the associativity of n-harmonic product $\hv$ by using (\ref{eq:2.6}) and 
the associativity of the harmonic product $*$. 
\begin{align*}
 w_{1} \hv ( w_{2} \hv w_{3} ) 
 &= S^{-1} \big( S( w_{1} ) \big) \hv \Big( S^{-1} \big( S( w_{2} ) \big) \hv S^{-1} \big( S( w_{3} ) \big) \Big) \\
 &= S^{-1} \big( S( w_{1} ) \big) \hv S^{-1} \big( S( w_{2} ) * S( w_{3} ) \big)  \\
 &= S^{-1} \Big( S( w_{1} ) * \big( S( w_{2} ) * S( w_{3} ) \big) \Big)  \\
 &= S^{-1}  \Big( \big( S( w_{1} ) * S( w_{2} ) \big) * S( w_{3} ) \Big)  \\
 &= S^{-1} \big( S( w_{1} ) * S( w_{2} ) \big) \hv S^{-1} \big( S( w_{3} ) \big) \\
 &= \Big( S^{-1} \big( S( w_{1} ) \big) \hv S^{-1} \big( S( w_{2} ) \big) \Big) \hv w_{3} \\
 &= (w_{1} \hv  w_{2}) \hv w_{3}. 
\end{align*}
This completes the proof. 
\end{proof}
\begin{lem}\label{2.10}
Let $\circ = * $ or $\sh$. A word $y^m w$ $(m \geq 0 , w \in \mathfrak{H}^0)$ of $\mathfrak{H}^1$ is represented uniquely by
\begin{align}
 y^m w &= w_{0} + w_{1} \cv y + w_{2} \cv y^{\cv 2} + \cdots + w_{m} \cv y^{\cv m} \quad (w_{i} \in \mathfrak{H}^0), 
\label{eq:2.8}
\end{align}
i.e., we have $\mathfrak{H}^0_{\cv}[y] \simeq \mathfrak{H}^1_{\cv}$. 
\end{lem}
\begin{proof}
We first prove that $y^m w$ can be represented as (\ref{eq:2.8}). 
By Corollary 5 of [IKZ], we have 
\[(y+x)^m S(w) = \sum_{i=0}^{m} v_{i} \circ y^{\circ i} \quad (v_{i} \in \mathfrak{H}^0). \;\; \] 
Using $(\ref{eq:2.6})$ or $(\ref{eq:2.7})$, we obtain  
\[\;\; y^m w = \sum_{i=0}^{m} S^{-1}(v_{i}) \cv y^{\cv i}.\]
(We have $S^{-1}(w_{1} w_{2}) = S_{2}(w_{1}) S^{-1}(w_{2})$ 
for $w_{1} \in \mathfrak{H}$, $w_{2} \in \mathfrak{H}^1$.) 
Therefore, the first assertion follows from $S^{-1}(\mathfrak{H}^0) \subset \mathfrak{H}^0$. 
We next prove the uniqueness of representation (\ref{eq:2.8}). 
We put
\[\qquad \sum_{i=0}^{m} w_{i} \cv y^{\cv i} = \sum_{i=0}^{m} v_{i} \cv y^{\cv i} \quad (w_{i},v_{i} \in \mathfrak{H}^0).\]
Using (\ref{eq:2.4}) or (\ref{eq:2.5}), we have 
\[\sum_{i=0}^{m} S(w_{i}) \circ y^{\circ i} = \sum_{i=0}^{m} S(v_{i}) \circ y^{\circ i}.\qquad \qquad\]
By $\mathfrak{H}^0_{\circ}[y] \simeq \mathfrak{H}^1_{\circ}$ (see [H2] and [R]), 
we have $S(w_{i}) = S(v_{i})$ for $i=0,1,\ldots,m$. 
So we have the second assertion. 
\end{proof}

\begin{prop}
We have two algebra homomorphisms
\[ \overline{Z}^{\hv}:\mathfrak{H}^{1}_{\hv} \longrightarrow \mathbb{R}[T] 
 \quad and \quad
   \overline{Z}^{\sv}:\mathfrak{H}^{1}_{\sv} \longrightarrow \mathbb{R}[T]\]
which are uniquely characterized by the properties that they both extend the n-evaluation map 
$\overline{Z}:\mathfrak{H}^0 \longrightarrow \mathbb{R}$ and send $y$ to $T$.
\label{2.11}
\end{prop}
\begin{proof}
The assertion follows because $\overline{Z}$ is homomorphism respect to $\hv$, $\sv$ and we have isomorphisms 
$\mathfrak{H}^0_{\hv}[y] \simeq \mathfrak{H}^1_{\hv}$, $\mathfrak{H}^0_{\sv}[y] \simeq \mathfrak{H}^1_{\sv}$. 
\end{proof}

The $\mathbb{Q}$-algebra homomorphisms $\overline{Z}^{\hv}$, $\overline{Z}^{\sv}$ have the following relations:
\[\overline{Z}^{\hv} = Z^{*} \circ S, \quad \overline{Z}^{\sv} = Z^{\sh} \circ S. \]
($\circ$ means composition) 
Indeed, $ Z^{*} \circ S $ and $Z^{\sh} \circ S$ satisfy the conditions of Proposition \ref{2.11}. 

\begin{thm}[Extended double shuffle relations of NMZVs]\label{2.12}
For $w_{1} \in \mathfrak{H}^1$ and $w_{2} \in \mathfrak{H}^0$, we have 
\[\overline{Z}^{\hv} ( w_{1} \sv w_{2} - w_{1} \hv w_{2} ) = 0
 \quad and \quad 
 \overline{Z}^{\sv} ( w_{1} \sv w_{2} - w_{1} \hv w_{2} ) = 0.\]
\end{thm}
\begin{proof}
By using (\ref{eq:2.4}), (\ref{eq:2.5}) and the relation $\overline{Z}^{\hv} = Z^{*} \circ S$, we have 
\begin{align*}
 \overline{Z}^{\hv} (w_{1} \sv w_{0} - w_{1} \hv w_{0}) 
   &= Z^{*} \circ S(w_{1} \sv w_{0} - w_{1} \hv w_{0})  \\
   &= Z^{*} \Big( S(w_{1}) \sh S(w_{0}) - S(w_{1}) * S(w_{0}) \Big) \\
   &= 0.
\end{align*}
The last equality follows by Theorem \ref{2.2} and the fact $S(\mathfrak{H}^1) \subset \mathfrak{H}^1$, 
$S(\mathfrak{H}^0) \subset \mathfrak{H}^0$. The other identity can be proven in the same way. 
\end{proof}

\section{Application}
In [H1], Hoffman proved the following theorem. 
\begin{thm}[{[}H1{]}]
For positive integers $k_{1}$, $k_{2}$, $\ldots$, $k_{n}$ and $k_{1} \geq 2$, we have 
\begin{eqnarray*}
\lefteqn{ \sum_{i=1}^{n} \zeta (k_{1}, \cdots, k_{i-1}, k_{i}+1, k_{i+1}, \cdots, k_{n} ) } \\
 & & \qquad \qquad 
   =\sum_{ \begin{subarray}{c} 1 \leq i \leq n \\ k_{i} \geq 2 \end{subarray} }
    \sum_{j=0}^{ k_{i} - 2 } \zeta (k_{1}, \cdots, k_{i-1}, k_{i}-j, j+1, k_{i+1}, \cdots, k_{n} ).
\end{eqnarray*}
\end{thm}
In this section, we prove an analogue of Hoffman's relations for NMZVs: 
\begin{thm}\label{3.3}
For positive integers $k_{1}$, $k_{2}$, $\ldots$, $k_{n}$ and $k_{1} \geq 2$, we have 
\begin{eqnarray*}
\lefteqn{ \sum_{i=1}^{n} \big(k_{i} - 1 + \delta_{ni}  \big) \,
 \zs (k_{1}, \cdots, k_{i-1}, k_{i}+1, k_{i+1}, \cdots, k_{n} ) } \\
 & & \qquad \qquad 
   =\sum_{ \begin{subarray}{c} 1 \leq i \leq n \\ k_{i} \geq 2 \end{subarray} }
    \sum_{j=0}^{ k_{i} - 2 } \zs (k_{1}, \cdots, k_{i-1}, k_{i}-j, j+1, k_{i+1}, \cdots, k_{n} ).
\end{eqnarray*}
\end{thm}

We first prove the following lemma. 
\begin{lem}\label{3.2}
$(i)$ Let $w$ be a word in $\mathfrak{H}$ and let k a positive integer. Then we have 
\begin{equation*}
y \sv z_{k} w = \begin{cases}
                \displaystyle z_{1} z_{k} + \sum_{j=0}^{k-2} z_{k-j} z_{j+1} - (k+1) z_{k+1} + z_{k} z_{1} 
                & \text{$( w =1 )$}, \\
                \displaystyle z_{1} z_{k} w + \sum_{j=0}^{k-2} z_{k-j} z_{j+1} w - k z_{k+1} w + z_{k} (y \sv w) 
                & \text{$( w \neq 1 )$} ,
                \end{cases}
\end{equation*}
where the summation is treated as $0$ when $k=1$. 

$(ii)$ For $k_{1}$, $k_{2}$, $\ldots$, $k_{n} \in \mathbb{Z}_{\geq 1}$, we have 
\begin{align}
y \sv z_{ k_{1} }  z_{ k_{2} } \cdots  z_{ k_{n} } 
 &= \sum_{i=0}^{n}  z_{ k_{1} }  \cdots z_{ k_{i} } z_{1} z_{ k_{i+1} } \cdots z_{ k_{n} } \nonumber\\
 & \quad + \sum_{ \begin{subarray}{c} 1 \leq i \leq n \\ k_{i} \geq 2 \end{subarray} }
  \sum_{j=0}^{ k_{i} - 2 } z_{ k_{1} } \cdots z_{ k_{i-1} } z_{ k_{i} - j} z_{j+1}  z_{ k_{i+1} } \cdots z_{ k_{n} }
   \label{eq:3.1}\\
 & \quad - \sum_{i=1}^{n} ( k_{i} + \delta_{ni} )
   z_{ k_{1} }  \cdots z_{ k_{i-1} } z_{ k_{i}+1 } z_{ k_{i+1} } \cdots z_{ k_{n} }. \nonumber
\end{align}
\end{lem}
\begin{proof}
(i) The case $k=1$ is clear from the definition of $\sv$. And we can prove the case $k \geq 2$ by induction on $k$. 
(ii) We prove the assertion by induction on $n$. The case $n=1$ follows from (i). 
We assume that the assertion is true for $n-1$. 
Using (i), we obtain 
\begin{align*}
 y \sv z_{k_{1}} z_{k_{2}} \cdots z_{k_{n}} 
 &= z_{1} z_{k_{1}} z_{k_{2}} \cdots z_{k_{n}} 
   + \sum_{j=0}^{k_{1}-2} z_{ k_{1}-j } z_{j+1} z_{k_{2}} \cdots z_{k_{n}} \\
 & \quad  - k_{1} z_{ k_{1}+1 } z_{k_{2}} \cdots z_{k_{n}} 
  + z_{ k_{1} } (y \sv z_{k_{2}} \cdots z_{k_{n}}) . 
\end{align*}
By the induction hypothesis, this expression equals the right hand side of (\ref{eq:3.1}). 
\end{proof}

\begin{proof}[Proof of Theorem $\ref{3.3}$]
By Lemma \ref{3.2} and the definition of $\hv$, we have 
\begin{align*}
\lefteqn{ y \sv z_{ k_{1} } z_{ k_{2} } \cdots z_{ k_{n} } 
         -   y \hv z_{ k_{1} } z_{ k_{2} } \cdots z_{ k_{n} } } \\
 &= \sum_{ \begin{subarray}{c} 1 \leq i \leq n \\ k_{i} \geq 2 \end{subarray} }
     \sum_{j=0}^{ k_{i} - 2 } z_{ k_{1} } \cdots z_{ k_{i-1} } z_{ k_{i} - j} z_{j+1}  z_{ k_{i+1} } \cdots z_{ k_{n} } \\
 & \quad - \sum_{i=1}^{n} ( k_{i} + \delta_{ni}-1 )
   z_{ k_{1} }  \cdots z_{ k_{i-1} } z_{ k_{i}+1 } z_{ k_{i+1} } \cdots z_{ k_{n} } .
\end{align*}
The right hand side is contained in $\mathfrak{H}^0$ by $k_{1} \geq 2$. 
Therefore, the assertion follows from Theorem \ref{2.12}. 
\end{proof}

\textsc{Graduate School of Mathematics, Kyushu University}

\textsc{Fukuoka 812-8581, Japan}

\textit{E-mail address}: muneta@math.kyushu-u.ac.jp


\end{document}